\documentclass[reqno, 11pt, final]{amsart}
\usepackage{bbding,textcomp,amsfonts,amsthm,amsmath,mathrsfs, amscd}
\usepackage{multirow, url}
\usepackage[utf8x]{inputenc}
\usepackage{appendix}
\usepackage{verbatim}
\usepackage{a4wide}

\newtheorem{theorem}{Theorem}

\newtheorem{conjecture}[theorem]{Conjecture}
\newtheorem{corollary}[theorem]{Corollary}
\newtheorem{lemma}[theorem]{Lemma}

\newcommand{\eps}{\varepsilon}

\newcommand{\rank}{\textrm{\upshape{rk}}\,}

\newcommand{\res}{\textrm{\upshape{res}}}

\newcommand{\ev}{\textrm{\upshape{ev}}}

\title{$\mathbb{F}_p$ is locally like $\mathbb{C}$}

\PrerenderUnicode{\unichar{355}}

\author{Codru\unichar{355} Grosu}
\thanks{This research was supported by the Deutsche Forschungsgemeinschaft within the research
training group `Methods for Discrete Structures' (GRK 1408).}

\address{Institut f\"{u}r Mathematik, Freie Universit\"at Berlin, Arnimallee 3-5, D-14195 Berlin, Germany}
\email{grosu.codrut@gmail.com}


\begin{document}

\begin{abstract}
Vu, Wood and Wood showed that any finite set $S$ in a characteristic zero integral domain can be mapped to $\mathbb{F}_p$, for infinitely many primes $p$, while preserving finitely many algebraic incidences of $S$. In this note we show that the converse essentially holds, namely any small subset of $\mathbb{F}_p$ can be mapped to some finite algebraic extension of $\mathbb{Q}$, while preserving bounded algebraic relations. This answers a question of Vu, Wood and Wood. We give several applications, in particular we show that for small subsets of $\mathbb{F}_p$, the Szemer\'edi-Trotter theorem holds with optimal exponent $4/3$, and we improve the previously best-known sum-product estimate in $\mathbb{F}_p$. We also give an application to an old question of R\'enyi. The proof of the main result is an application of elimination theory and is similar in spirit with the proof of the quantitative Hilbert Nullstellensatz.
\end{abstract}
\maketitle

\section{Introduction}
\label{sec:intro}

Suppose $p$ is a prime and $N$ a positive integer. In what follows $\mathbb{Z}_N$ denotes the additive group of integers modulo $N$, $\mathbb{F}_p$ the field with $p$ elements and $\mathbb{Z}_{(p)}$ the localization of $\mathbb{Z}$ at $(p)$, which is the same as the ring of fractions with denominator not divisible by $p$.

Let $k \geq 1$ be an integer, $Z$ and $W$ two abelian groups and $A \subseteq Z, B \subseteq W$ finite subsets. A bijection $\phi : A \rightarrow B$ is a \textit{Freiman isomorphism of order $k$}, or simply $F_k$-isomorphism, if for any $a_1, \ldots, a_{2k} \in A$ we have
\begin{equation*}
a_1 + \ldots + a_k = a_{k+1} + \ldots + a_{2k}
\end{equation*}
if and only if
\begin{equation*}
\phi(a_1) + \ldots + \phi(a_k) = \phi(a_{k+1}) + \ldots + \phi(a_{2k}).
\end{equation*}
From the definition it follows that any $F_{k+1}$-isomorphism is also an $F_k$-isomorphism, and furthermore translation does not affect the isomorphism. An important property in additive combinatorics is that any finite subset of a torsion-free group is $F_k$-isomorphic to a subset of $\mathbb{Z}_N$, for any large enough $N$. This helps reducing general additive problems to $\mathbb{Z}_p$, where more powerful techniques, such as Fourier analysis, are available.

In the other direction it is well-known that small subsets of $\mathbb{Z}_p$, with $p$ prime, are Freiman isomorphic to subsets of $\mathbb{Z}$.
\begin{theorem}[\cite{Bilu98}]
\label{thm:zp}
Let $A \subseteq \mathbb{Z}_p$, where $p$ is a prime. If $|A| \leq \log_{2k}p$, then there exists a set of integers $A' \subset \mathbb{Z}$ such that the canonical homomorphism $\mathbb{Z} \rightarrow \mathbb{Z}_p$ induces an $F_k$-isomorphism of $A'$ onto $A$.
\end{theorem}
The result holds for $|A| \leq \log_{2k}p + \log_{2k}\log_{2k}p$ as well. In \cite{Bilu98} it is also shown the existence of a set $A \subset \mathbb{Z}_p$ of cardinality at most $2\log_k p +1$ which is not $F_k$-isomorphic to any set of integers. Assuming $A$ has small doubling constant allows the theorem to hold for $|A| \leq cp$, for some $c> 0$. This is the \textit{Freiman rectification principle} (see \cite{Bilu98}, \cite{Green06}).

It is now a natural question if it is possible to preserve both the additive and multiplicative structure. In this direction we have the following result of Vu, Wood and Wood. 
\begin{theorem}[\cite{VuWood}]
\label{thm:vuwood}
Let $S$ be a finite subset of a characteristic zero integral domain $D$, and let $L$ be a finite set of non-zero elements in the subring $\mathbb{Z}[S]$ of $D$. There exists an infinite sequence of primes with positive relative density such that for each prime $p$ in the sequence, there is a ring homomorphism 
$\phi_p : \mathbb{Z}[S]\rightarrow \mathbb{F}_p$ satisfying $0 \notin \phi_p(L)$.
\end{theorem}
Here $\mathbb{Z}[S]$ is the smallest subring of $D$ containing $S$. 

It was asked by Vu, Wood and Wood \cite{VuWood} whether given a small enough set $A \subseteq \mathbb{F}_p$, it is possible to map $A$ to some characteristic zero integral domain, while preserving algebraic incidences. 

Let us first make a few observations. One could not always map $A$ to $\mathbb{Z}$, as we may need, for example, to preserve identities of the form $y^2 + z^2 = 0$ with $y, z \neq 0$ for some $y, z \in A$. Also, we should allow only "bounded" algebraic incidences, as any identity of the form $py = 0$ with non-zero $y \in A$ can not be mapped in any characteristic zero integral domain. Therefore the following definitions make sense.

Let $k, t > 0$. A polynomial $f \in \mathbb{Z}[x_1, \ldots, x_n]$ is called $(k, t)$-bounded if $\|f\|_1 \leq k$, and its degree is at most $t$. Here $\|f\|_1$ represents the sum of the absolute values of the coefficients of $f$, and similarly we define $\|f\|_\infty$ to be the maximum of the absolute values of the coefficients of $f$. When we evaluate $f$ at a point $(a_1, \ldots, a_n)$ with $a_i \in R$ for some ring $R$, all operations are carried in $R$, in the natural way. If $k = t$, we simply call $f$ $k$-bounded. If $t=1$, we say $f$ is a $k$-bounded linear polynomial.

Now let $R_1, R_2$ be two rings and $A \subseteq R_1, B \subseteq R_2$ finite subsets. We call a bijection $\phi : A \rightarrow B$ a \textit{Freiman ring-isomorphism of order $k$}, or simply $F_k$-ring-isomorphism, if $A = \{a_1, \ldots, a_n\}$ and for any $k$-bounded $f \in \mathbb{Z}[x_1, \ldots, x_n]$ we have
\begin{equation*}
f(a_1, \ldots, a_n) = 0
\end{equation*}
if and only if
\begin{equation*}
f(\phi(a_1), \ldots, \phi(a_n)) = 0.
\end{equation*}

Our main result is the following.

\begin{theorem}
\label{thm:main}
Let $k \geq 2$ be an integer, $p$ be a prime and $A \subseteq \mathbb{F}_p$. If $|A| < \log_2\log_{2k}\log_{2k^2}p~-~1$ then there exists a finite algebraic extension $K$ of $\mathbb{Q}$ of degree at most $(2k)^{2^{|A|}}$, a subset $A' \subset K$ and a homomorphism $\phi_p : \mathbb{Z}[A'] \rightarrow \mathbb{F}_p$ such that $\phi_p$ is an $F_k$-ring-isomorphism between $A'$ and $A$.
\end{theorem}

One can use the construction from \cite{Bilu98} to see that for any $k \geq 2$ and any prime number $p$ there exists a subset $A \subseteq \mathbb{F}_p$ of size $O(\log p)$, which is not $F_k$-ring-isomorphic to any subset of a characteristic zero integral domain. 
For $k \geq 3$ we can improve this bound to the following.

\begin{theorem}
\label{thm:sharp}
For any $k \geq 3$ and any prime number $p \geq 2^{32(k-1)^2\log_2^2(16(k-1))}$ there exists a subset $A \subseteq \mathbb{F}_p$ of size $|A| \leq \frac{10}{k-1}\frac{\log_{2} p}{\log_{2}\log_{2} p}$ which is not $F_k$-ring-isomorphic to any subset of a characteristic zero integral domain.
\end{theorem}

It is an open problem if a better bound is possible. In this direction I would like to make the following conjecture.

\begin{conjecture}
\label{conj:sharp}
For any $k \geq 3$ there is an infinite sequence of prime numbers, such that for each prime $p$ in the sequence, there exists a subset $A \subseteq \mathbb{F}_p$ of size $O(\log \log p)$ which is not $F_k$-ring-isomorphic to any subset of a characteristic zero integral domain.
\end{conjecture}

As explained in Section 5, this conjecture would have a positive answer if, for example, there are infinitely many Mersenne primes (primes of the form $2^n -1$; this would follow from the Lenstra–Pomerance–Wagstaff conjecture), or infinitely many Fermat primes (primes of the form $2^{2^n}+1$; this is a question of Eisenstein).

The proof of Theorem \ref{thm:main} uses elimination theory. This is not the first time when elimination theory is applied to additive combinatorics: similar techniques were used by Chang in the proof of Lemma $2.14$ from \cite{chang03}. We state this lemma below in an equivalent form.
\begin{lemma}[Lemma $2.14$, \cite{chang03}]
\label{lem:chang}
Let $f_1, \ldots, f_s \in \mathbb{Z}[x_1, \ldots, x_n]$ be polynomials of degree at most $t$ and $\|\cdot\|_\infty$-norm at most $k$. If the system
\begin{equation*}
f_1(x) = \ldots = f_s(x) = 0
\end{equation*}
has a solution $(a_1, \ldots, a_n) \in \mathbb{C}^n$, then it also has a solution $(b_1, \ldots, b_n)$, where each $b_i$ is the root of an integer polynomial of degree at most $C$ and $\|\cdot\|_\infty$-norm at most $Ck^C$, with $C := C(t, n, s)$ depending only on $t, n$ and $s$.
\end{lemma}
This lemma is discussed by Tao on his blog \cite{tao0711}, in particular he gives a proof of it using nonstandard analysis. Neither this proof nor the proof in \cite{chang03} provides a bound on the constant $C$.

The proof of Lemma \ref{lem:chang} from \cite{chang03} shows in fact a bit more; namely that if we are further given a polynomial $g \in \mathbb{Z}[x_1, \ldots, x_n]$ which does not vanish at $(a_1, \ldots, a_n)$, and has degree at most $t$ and $\|\cdot\|_\infty$-norm at most $k$, then it is possible to choose $(b_1, \ldots, b_n)$ such that $g(b_1, \ldots, b_n) \neq 0$. On close examination of the proof it turns out that translated into the correct setting it implies the following weak version of Theorem \ref{thm:main}.

\begin{theorem}
\label{thm:weakmain}
For any $k \geq 2$ there exists a function $\nu_k : \mathbb{N} \rightarrow \mathbb{N}$ with $\lim_{n \rightarrow \infty} \nu_k(n) = \infty$, such that the following holds. If $p$ is a prime and $A \subseteq \mathbb{F}_p$ with $|A| \leq \nu_k(p)$ then there exists a finite algebraic extension $K$ of $\mathbb{Q}$ and a subset $A' \subset K$ such that $A'$ is $F_k$-ring-isomorphic with $A$.
\end{theorem}

An upper bound for the constant $C$ implies a lower bound for $\nu_k(n)$; however from the proof of Lemma \ref{lem:chang} one can only extract a rather poor bound for $C$. 

It is also important to note that Theorem~\ref{thm:weakmain} does not provide any bound on the degree of the field extension $K$, nor does it guarantee that the $F_k$-ring-isomorphism is the restriction of a genuine ring homomorphism, as in Theorem~\ref{thm:main}. In fact it is easy to construct an example of a Freiman ring-isomorphism $\phi$ between a subset $A' \subset \mathbb{C}$ and a subset $A \subset \mathbb{F}_p$ such that $\phi$ is not the restriction of any ring homomorphism between $\mathbb{Z}[A']$ and $\mathbb{F}_p$. Indeed, consider $A' := \{-\frac{1}{2}, 2\} \subset \mathbb{C}$ and $A := \{3, 7\} \subset \mathbb{F}_{11}$. The map $\phi$ sending $-\frac{1}{2}$ to $3$ and $2$ to $7$ is an $F_2$-ring-isomorphism, but it is obviously not the restriction of a ring homomorphism between $\mathbb{Z}[A']$ and $\mathbb{F}_{11}$ (as any such homomorphism would send $2$ to $2$). Examples for arbitrarily large $k$ and $p$ can be constructed as well.

The rest of the paper is organized as follows. 

We start by giving several applications of the main result to subsets of $\mathbb{F}_p$ of size $O(\log \log \log p)$. In Section 2, we use Theorem \ref{thm:main} to prove a Szemer{\' e}di-Trotter type theorem with optimal exponent $4/3$. In Section 3, we apply Theorem \ref{thm:main} to improve the currently best-known sum-product estimate in $\mathbb{F}_p$. Finally, in Section 4 we give several estimates for sets with small doubling constant. All these results are proved by transferring the corresponding theorem from $\mathbb{C}$ to $\mathbb{F}_p$ via Theorem~\ref{thm:main}. In all these applications only the existence of a Freiman ring-isomorphism between $A$ and a subset of $\mathbb{C}$ is needed, and not the stronger conclusion of Theorem~\ref{thm:main}.

In Section 5 we give an application of Theorem~\ref{thm:main} to an old question of R\'enyi. In this case we will make essential use of the upper bound on the degree of the algebraic extension $K$ in Theorem~\ref{thm:main}.

In Section 6 we show, as an example for the general strategy, how to preserve bounded linear polynomials. In Section 7 we gather all the necessary results from elimination theory. Finally, Section 8 is devoted to the proof of Theorem \ref{thm:main} and in Section 9 we prove Theorem \ref{thm:sharp}.

We conclude with some further remarks concerning Freiman isomorphisms and Lemma \ref{lem:chang}.

\medskip
\textbf{Remark.} After completion of this work I was informed by Pierre Simon that one can use the arithmetic Nullstellensatz stated in \cite{krick01} to prove a good lower bound for the function $\nu_k$ in Theorem~\ref{thm:weakmain}. With his idea, my own computations show that one can take $\nu_k(p) = \Omega(\frac{\log\log p}{\log \log \log p})$. This would improve the upper bound for $n$ in Theorems \ref{thm:mysztr}, \ref{thm:sumpr}, \ref{thm:mysmalldb} and \ref{thm:inverse} below to $O(\frac{\log\log p}{\log \log \log p})$.

Moreover, in his blog post \textit{Rectification and the Lefschetz principle} \cite{tao0313}, Tao presented a short proof of the following version of Theorem~\ref{thm:main}.
\begin{theorem}
\label{thm:tao}
Let $k, n \geq 1$. If $\mathbb{F}$ is a field of characteristic at least $C_{k, n}$ for some $C_{k, n}$ depending only on $k$ and $n$, and $A$ is a subset of $\mathbb{F}$ of cardinality $n$, then there exists a map $\phi:A \rightarrow A'$ into a subset $A'$ of the complex numbers which is a Freiman ring-isomorphism of order $k$.
\end{theorem}
The proof uses non-standard analysis, and hence does not offer any bound on $C_{k, n}$. However, unlike Theorem \ref{thm:main}, it also applies to fields of prime power order.

\section{The Szemer\'edi-Trotter theorem}

The well-known Szemer\'edi-Trotter theorem gives a tight upper bound on the number of incidences between a finite set of lines and a finite set of points in $\mathbb{R} \times \mathbb{R}$. This was extended to the complex plane $\mathbb{C}^2$ by T\'oth.  
\begin{theorem}[\cite{Toth03}]
\label{thm:sztrotter}
Let $\mathcal{P}$ and $\mathcal{L}$ be sets of points and lines in $\mathbb{C}^2$, with cardinalities $|\mathcal{P}|, |\mathcal{L}| \leq n$. Then there is a positive absolute constant $c$ such that  
\begin{equation*}
|\{(p, l) \in \mathcal{P} \times \mathcal{L} : p \in l\}| \leq cn^{4/3}.
\end{equation*}
\end{theorem}
T\'oth's paper is still unpublished; but very recently Zahl gave a different proof of Theorem \ref{thm:sztrotter} in \cite{zahl12}. Unfortunately, Zahl's paper is also still unpublished. However, if we allow an $\eps>0$ error in the exponent, and the constant $c$ to depend on $\eps$, then in this form Theorem \ref{thm:sztrotter} follows from a generalization of the Szemer\'edi-Trotter theorem to algebraic varieties due to Solymosi and Tao \cite{tao12}.

The problem of establishing a similar bound in $\mathbb{F}_p$ has been considered before (\cite{tao04}, \cite{helfgott11}). We have the following result, due to Helfgott and Rudnev.
\begin{theorem}[\cite{helfgott11}]
\label{thm:sztrotterP}
Let $p$ be a prime number, and $\mathcal{P}$ and $\mathcal{L}$ sets of points and lines in $\mathbb{F}_p^2$, with $|\mathcal{P}|, |\mathcal{L}| \leq n$ and $n < p$. Then there is a positive absolute constant $c$ such that  
\begin{equation*}
|\{(p, l) \in \mathcal{P} \times \mathcal{L} : p \in l\}| \leq cn^{\frac{3}{2}-\delta},
\end{equation*}
with $\delta = \frac{1}{10678}$.
\end{theorem}
The best (still unpublished) bound to date for $n < p$ is due to Jones \cite{Jones12}, who proved that one can take $\delta = \frac{1}{662}-o(1)$ in the above.

We show that one can achieve optimal exponent $4/3$ in Theorem \ref{thm:sztrotterP} provided $n$ is sufficiently small compared to $p$.
\begin{theorem}
\label{thm:mysztr}
Let $p$ be a prime number, and $\mathcal{P}$ and $\mathcal{L}$ sets of points and lines in $\mathbb{F}_p^2$, with $|\mathcal{P}|, |\mathcal{L}| \leq n$ and $5n < \log_2\log_6\log_{18} p - 1$. Then there is a positive absolute constant $c$ such that  
\begin{equation*}
|\{(p, l) \in \mathcal{P} \times \mathcal{L} : p \in l\}| \leq cn^{4/3}.
\end{equation*}
Moreover, this inequality is sharp up to the constant $c$.
\end{theorem}
\begin{proof}
We may assume w.l.o.g. that $|\mathcal{P}| = |\mathcal{L}| = n$, by adding some points and lines if necessary. Let $\mathcal{P} = \{(x_i, y_i) : 1 \leq i \leq n\}$. By uniquely parametrizing each line $l \in \mathcal{L}$ defined by $a_iy+b_ix+c_i = 0$, by the ordered triple $(a_i, b_i, c_i)$, let $\mathcal{L} = \{(a_i, b_i, c_i) : 1 \leq i \leq n\}$. Now form the set $A:=\cup_{i=1}^n\{x_i, y_i, a_i, b_i, c_i\}$. As $|A| \leq 5n$, we may apply Theorem \ref{thm:main} to find a subset $A' \subset \mathbb{C}$ and an $F_3$-ring-isomorphism $\phi$ between $A$ and $A'$. By definition we have 
\begin{equation*}
a_jy_i + b_jx_i+c_j = 0 \Leftrightarrow \phi(a_j)\phi(y_i)+\phi(b_j)\phi(x_i)+\phi(c_j) = 0, \forall 1 \leq i, j \leq n,
\end{equation*}
hence the number of incidences between $\mathcal{P}$ and $\mathcal{L}$ in $\mathbb{F}_p^2$ is the same as the number of incidences between $\phi(\mathcal{P})$ and $\phi(\mathcal{L})$ in $\mathbb{C}$. Note that $\phi(\mathcal{P})$ and $\phi(\mathcal{L})$ have cardinality exactly $n$ as $\phi$ is bijective. Hence by Theorem \ref{thm:sztrotter}, the number of incidences is $O(n^{4/3})$, as desired.

To show that the bound is sharp, we use a standard construction that proves sharpness of the Szemer\'edi-Trotter theorem in $\mathbb{R}^2$. Let $r := \lfloor \frac{1}{2}n^{1/3} \rfloor$. We set $\mathcal{P}$ to be the points of the lattice $[r] \times [2r^2]$ in $\mathbb{F}_p^2$, and $\mathcal{L}$ to be all lines $y = mx + b$, with $1 \leq m \leq r, 1 \leq b \leq r^2$. Then every line from $\mathcal{L}$ is incident with exactly $r$ points from $\mathcal{P}$, for a total of $r^4 = \Theta(n^{4/3})$ incidences. 
\end{proof}

One can now combine Theorem \ref{thm:mysztr} with Theorem \ref{thm:vuwood} to generalize Theorem \ref{thm:sztrotter} to any characteristic zero integral domain. As this statement can be proved directly with no recurse to Theorem \ref{thm:mysztr}, we do not discuss it here (see Theorem $2.3$ and Lemma $7.1$ from \cite{VuWood} for more details).

\section{Sum-product estimates in $\mathbb{F}_p$}

Suppose $R$ is a commutative ring and $A \subset R$ a finite subset. We can define the sumset $A+A := \{a+b : a, b \in A\}$ and the product $A \cdot A := \{a b : a, b \in A\}$. Intuitively, the quantities $|A+A|$ and $|A \cdot A|$ can not both be small. The prototype theorem is a lower bound of the form $\max\{|A+A|, |A \cdot A|\} \geq c|A|^{1+\eps_{R}}$, where $c>0$ is an absolute constant and $\eps_{R}$ depends on the ring $R$. The first sum-product estimate is due to Erd{\H o}s and Szemer\'edi \cite{erdosSz83} for the case $R = \mathbb{Z}$ and it was followed by numerous improvements and generalizations (\cite{elekes97}, \cite{nathanson97}, \cite{ford98}, \cite{chang05}, \cite{solymosi09}). For $R = \mathbb{C}$, the best-known value $\eps_{\mathbb{C}} = \frac{3}{11}-o(1)$ was for many years given by a result of Solymosi \cite{solymosi05}. Using a beautiful geometric argument, Konyagin and Rudnev \cite{konyrud} have very recently improved this to $\eps_{\mathbb{C}} = \frac{1}{3} - o(1)$, thus matching the lower bound for the reals.
\begin{theorem}[\cite{konyrud}]
\label{thm:konyrud}
Suppose $A \subset \mathbb{C}$. Then there is a positive absolute constant $c$ such that
\begin{equation}
\label{eq:sumC}
|A+A|+|A\cdot A| \geq c |A|^{1+\frac{1}{3}-o(1)}.
\end{equation}
\end{theorem}

Bourgain, Katz and Tao \cite{tao04} showed that a sum-product theorem holds in $\mathbb{F}_p$.  Substantial work has gone into finding the best value for $\eps_{\mathbb{F}_p}$. Garaev \cite{garaev07} showed that for $|A| < \sqrt{p}$ one can take $\eps_{\mathbb{F}_p} = \frac{1}{14} - o(1)$. Katz and Shen \cite{katz08} improved this to $\frac{1}{13} - o(1)$, and then Bourgain and Garaev \cite{bourgain08} showed that $\frac{1}{12} - o(1)$ is in fact possible. Li \cite{li11} later removed the $o(1)$ term. The best result to date is due to Rudnev \cite{rudnev11}, who showed that
\begin{equation}
\label{eq:sumZp}
|A+A|+|A\cdot A| \geq c |A|^{1+\frac{1}{11}-o(1)},
\end{equation}
whenever $|A| < \sqrt{p}$.

We now improve \eqref{eq:sumZp} for small $A$.
\begin{theorem}
\label{thm:sumpr}
Let $p$ be a prime number and $A \subseteq \mathbb{F}_p$ with $|A| < \log_2\log_8\log_{32} p - 1$. Then 
\begin{equation*}
|A+A|+|A\cdot A| \geq c |A|^{1+\frac{1}{3}-o(1)},
\end{equation*}
for some positive absolute constant $c$.
\end{theorem}
\begin{proof}
We apply Theorem \ref{thm:main} to find a subset $A' \subset \mathbb{C}$ and an $F_4$-ring-isomorphism $\phi$ between $A$ and $A'$. Then $|\phi(A)+\phi(A)| = |A+A|$ and $|\phi(A) \cdot \phi(A)| = |A \cdot A|$. By \eqref{eq:sumC} applied to $A' = \phi(A)$, the theorem follows.
\end{proof}

\section{Estimates for sets with small doubling constant}

We gather in this section several miscellaneous results for the case when $A$ has small doubling constant. We first have the following result, due to Solymosi.
\begin{theorem}[\cite{solymosi05}]
If $A \subset \mathbb{C}$ and $|A| = n$ with $|A + A| \leq Cn$, then $|A \cdot A| \geq cn^2 /\log n$.
\end{theorem}

This transfers immediately to $\mathbb{F}_p$ as follows.
\begin{theorem}
\label{thm:mysmalldb}
If $A \subseteq \mathbb{F}_p$ and $|A| = n < \log_2\log_8\log_{32} p - 1$ with $|A + A| \leq Cn$, then $|A \cdot A| \geq cn^2 /\log n$.
\end{theorem}
The proof is similar to that of Theorem \ref{thm:sumpr} and we omit it. We also have the following result due to Chang \cite{chang03}.
\begin{theorem}
\label{thm:chang}
Let $A \subset \mathbb{C}$ with $|A| = n$ and $|A + A| \leq Cn$, for some $C > 0$. Then the following holds.
\begin{itemize}
\item[(i)] If $0 \notin A$ then $|A^{-1} + A^{-1}| > \exp^{-C'\frac{\log n}{\log \log n}}n^2$, for some $C'$ depending only on $C$.
\item[(ii)] If $f(x) \in \mathbb{C}[x]$ is a polynomial of degree $t \geq 2$ then $|f(A) + f(A)| > \exp^{-C'\frac{\log n}{\log \log n}}n^2$, for some $C' := C'(C, t)$.
\end{itemize}
\end{theorem}
Here $A^{-1} = \{a^{-1} : a \in A\}$ and $f(A) = \{f(a) : a \in A\}$. The proof of Theorem \ref{thm:chang} uses algebraic methods, in particular Lemma \ref{lem:chang}, but also relies crucially on facts specific to $\mathbb{C}$. We now transfer this theorem to small subsets of $\mathbb{F}_p$.
\begin{theorem}
\label{thm:inverse}
Let $A \subseteq \mathbb{F}_p$ with $|A| = n$ and $|A + A| \leq Cn$, for some $C > 0$. Then the
following holds.
\begin{itemize}
\item[(i)] Suppose $2n < \log_2\log_8\log_{32} p - 1$ and $0 \notin A$. Then $|A^{-1} + A^{-1}| > \exp^{-C'\frac{\log n}{\log \log n}}n^2$, for some $C'$ depending only on $C$.
\item[(ii)] Let $f(x) \in \mathbb{Z}[x]$ be a $k$-bounded polynomial of degree at least $2$. If $n < \log_2\log_{8k}\log_{32k^2} p - 1$ then $|f(A) + f(A)| > \exp^{-C'\frac{\log n}{\log \log n}}n^2$, for some $C' := C'(C, k)$.
\end{itemize}
\end{theorem}
\begin{proof}
We first prove (i).

We apply Theorem \ref{thm:main} to find a subset $A' \subset \mathbb{C}$ and an $F_4$-ring-isomorphism $\phi$ between $A \cup A^{-1}$ and $A'$. Then $|\phi(A)| = n$, $|\phi(A)+\phi(A)| = |A+A|$ and $|\phi(A^{-1}) + \phi(A^{-1})| = |A^{-1}+A^{-1}|$. Moreover, all identities of the form $a^{-1} a = 1, a \in A,$ must be preserved by the ring-isomorphism, and hence $\phi(a^{-1}) = \phi(a)^{-1}, \forall a \in A$. Then by applying Theorem \ref{thm:chang}, (i), the result follows.

We now prove (ii). 

We apply Theorem \ref{thm:main} to find a subset $A' \subset \mathbb{C}$ and an $F_{4k}$-ring-isomorphism $\phi$ between $A$ and $A'$. Then $|\phi(A)| = n$ and $|\phi(A)+\phi(A)| = |A+A|$. We further have 
\begin{equation*}
f(\phi(a)) + f(\phi(b)) - f(\phi(c)) - f(\phi(d)) = 0 \Leftrightarrow f(a)+f(b)-f(c)-f(d) = 0,
\end{equation*}
for any $a, b, c, d \in A$, as $\phi$ is an $F_{4k}$-ring-isomorphism. Hence $|f(\phi(A)) + f(\phi(A))| = |f(A)+f(A)|$. Then by applying Theorem \ref{thm:chang}, (ii), the result follows.
\end{proof}

\section{A question of R\'enyi}
\label{sec:renyi}

Let $K$ be a field of characteristic zero. For a polynomial $f \in K[x]$ we define $N(f)$ to be the number of non-zero terms of $f$. For $k \geq 1$, let
\begin{equation}
\label{eq:square}
Q_K(k) = \min_{f \in K[x] : N(f) = k}N(f^2).
\end{equation}
As reported by Erd\H os \cite{erdos49}, it was first asked by R\'edei if $Q_{\mathbb{R}}(k) < k$ is possible, and R\'enyi \cite{renyi47} later constructed an example showing $Q_{\mathbb{Q}}(29) \leq 28$. R\'enyi made several conjectures about the behaviour of $Q_{\mathbb{R}}(k)$.

He conjectured that $\lim_{k \rightarrow \infty}\frac{Q_{\mathbb{R}}(k)}{k} = 0,$ and this was proved by Erd\H os \cite{erdos49}, who in fact showed that $Q_{\mathbb{Q}}(k) < c k^{1-\eps}$, for some positive absolute constants $c$ and $\eps$.

R\'enyi further conjectured that $\lim_{k \rightarrow \infty}Q_{\mathbb{R}}(k) = \infty$, and this was proved many years later by Schinzel \cite{schinzel87}, using a very ingenious argument. Schinzel showed that $Q_{K}(k) \geq c\log \log k$, for some positive absolute constant $c$ and any field $K$ of characteristic zero. This lower bound was not improved for another 20 years, until recently Schinzel and Zannier \cite{zannier09}, by an adaptation of the original method of Schinzel, proved that $Q_K(k) \geq c\log k$, for some positive absolute constant $c$.

Erd\H os \cite{erdos49} asked for the determination of the order of $Q_{\mathbb{R}}(k)$, and the general belief seems to be that $Q_{\mathbb{R}}(k)$ should be closer to the upper bound than the lower bound. Despite some work in this direction (\cite{verdenius49}, \cite{coppersmith91}), a solution to this problem seems at present out of reach.

From the definition we see that for any $k \geq 1$,
\begin{equation}
\label{eq:renyi}
Q_{\mathbb{C}}(k) \leq Q_{\mathbb{R}}(k) \leq Q_{\mathbb{Q}}(k).
\end{equation}
It is less known that R\'enyi \cite{renyi47} (see also \cite{erdos49}) asked whether equality holds in \eqref{eq:renyi} everywhere for any $k$, and this problem seems to have received little attention. 

For any $k \geq 1$ it also holds that
\begin{equation}
\label{eq:renyi2}
Q_{\mathbb{C}}(k) \leq Q_{K}(k) \leq Q_{\mathbb{Q}}(k),
\end{equation}
for any finite algebraic extension $K$ of $\mathbb{Q}$, and thus if we have equality in \eqref{eq:renyi}, then we also have equality in \eqref{eq:renyi2}. In view of this we have the following result.

\begin{theorem}
\label{thm:myrenyi}
For any $k \geq 3$ there exists a finite algebraic extension $K$ of $\mathbb{Q}$ such that $Q_{\mathbb{C}}(k) = Q_K(k)$, with degree at most $k^{2^k}$, if $k$ is even, and at most $(k+1)^{2^k}$, if $k$ is odd.
\end{theorem}
\begin{proof}
Set $s := \left \lfloor \frac{k+1}{2} \right \rfloor$. Note that $s \geq 2$.

Let $f \in \mathbb{C}[x]$ be a polynomial with $k$ non-zero terms minimizing $N(f^2)$. Suppose $f = a_0+a_1x^{n_1} + \ldots + a_{k-1}x^{n_{k-1}}$ and set $A := \{a_0, \ldots, a_{k-1}\} \subset \mathbb{C}$.

We now apply Theorem \ref{thm:vuwood} in order to find a sufficiently large prime $p$ (compared to $k$) and a homomorphism $\phi : \mathbb{Z}[A] \rightarrow \mathbb{F}_p$ which is an $F_s$-ring-isomorphism between $A$ and $\phi(A)$. We then apply Theorem \ref{thm:main} to the set $\phi(A)$ in order to find a finite algebraic extension $K$ of $\mathbb{Q}$ of degree at most $(2s)^{2^k}$, a subset $B \subset K$ and a map $\psi$ between $\phi(A)$ and $B$, which is an $F_s$-ring-isomorphism. Then $\psi \circ \phi$ is an $F_s$-ring-isomorphism between $A$ and $B$ by construction.

Let $g = (\psi \circ \phi)(a_0) + (\psi \circ \phi)(a_1)x^{n_1} + \ldots + (\psi \circ \phi)(a_{k-1})x^{n_{k-1}}$. Then $g \in K[x]$ and $N(g) = k$. As any coefficient of $g^2$ is given by a polynomial with integer coefficients of degree at most $2$ and $\|\cdot\|_1$-norm at most $s$, evaluated at $((\psi \circ \phi)(a_0), (\psi \circ \phi)(a_1), \ldots, (\psi \circ \phi)(a_{k-1}))$, we see that $N(g^2) = N(f^2)$. Consequently $Q_K(k) \leq N(f^2) = Q_{\mathbb{C}}(k)$, thus proving the theorem.
\end{proof}

\textbf{Remark.} Lemma \ref{lem:multipl} below shows that $K$ can in fact be chosen of degree at most $4^{2^k}$.
 
\section{Preserving the additive structure}
\label{sec:additive}

For comparison reasons we start by sketching a proof of Theorem \ref{thm:zp}, following \cite{Bilu98}.
\begin{proof}[Proof of Theorem \ref{thm:zp}]
We first choose $0 < t < p$ such that multiplying every element of $A$ by $t$ (modulo $p$) results in a set $A^{*} \subseteq \{-\left\lfloor\frac{p}{2k} \right\rfloor, \ldots, \left\lfloor\frac{p}{2k} \right\rfloor\}$. The existence of $t$ follows from the Kronecker approximation theorem (Corollary 3.2.5, \cite{TaoVuA}). Let $m \in \mathbb{Z}$ be such that $mt \equiv 1\,(\textrm{mod } p)$. We multiply every element of $A^{*}$ by $m$ to obtain $A'$. Then the canonical homomorphism maps $A'$ onto $A$, and one easily sees that this is also an $F_k$-isomorphism.
\end{proof}

We will now consider the problem of preserving bounded linear polynomials. As we allow non-zero constant terms, we will have to find a proof different from that of Theorem \ref{thm:zp}.

We first prove an inequality.

\begin{lemma}
\label{lem:ineq}
Suppose $M = (m_{ij})$ is an $n \times n$ matrix with entries $m_{ij} \in \mathbb{Z}[x_1, \ldots, x_r]$. 
If for any $i$, $\sum_{j} \|m_{ij}\|_1 \leq k,$ then $\|\det(M)\|_1 \leq k^n$. Furthermore, for any matrix $M$ with integer entries, $|\det(M)|$ is at most the product of the $\|\cdot\|_1$-norms of the rows.
\end{lemma}
\begin{proof}
We use the easily verified inequality $\|fg\|_1 \leq \|f\|_1 \|g\|_1$, which holds for any $f, g \in \mathbb{Z}[x_1, \ldots, x_r]$, to see that 
\begin{align*}
\|\det(M)\|_1 &\leq \sum_{\pi \in S_n} \|m_{1\pi(1)}\|_1\ldots\|m_{n\pi(n)}\|_1 
	\leq \sum_{1 \leq i_1, \ldots, i_n \leq n} \|m_{1i_1}\|_1\ldots\|m_{ni_n}\|_1 \\
	&= (\sum_j \|m_{1j}\|_1) \ldots (\sum_j \|m_{nj}\|_1) 
	\leq k^n. 
\end{align*}
\end{proof}
The last statement of Lemma \ref{lem:ineq} is also a consequence of Hadamard's inequality. 

We now have the following technical result.
\begin{lemma}
\label{lem:additive}
Let $k > 1$ be an integer and $p$ be a prime. Suppose $A = \{a_1, \ldots, a_n\} \subseteq \mathbb{Z}_p$ and let $\mathcal{L}_1, \mathcal{L}_2 \subset \mathbb{Z}[x_1, \ldots, x_n]$ be collections of $k$-bounded linear polynomials, such that any $f \in \mathcal{L}_1$ is zero when evaluated at $(a_1, \ldots, a_n)$, and any $f \in \mathcal{L}_2$ is non-zero when evaluated at $(a_1, \ldots, a_n)$. If
$|A| < \log_k p - 1$,
then there exists $A' = \{b_1, \ldots, b_n\} \subset \mathbb{Z}_{(p)}$ such that the canonical homomorphism $\mathbb{Z}_{(p)} \rightarrow \mathbb{Z}_p$ maps $b_i$ to $a_i$, and $f(b_1, \ldots, b_n) = 0$ for $f \in \mathcal{L}_1$, $f(b_1, \ldots, b_n) \neq 0$ for $f \in \mathcal{L}_2$.
\end{lemma}
This directly implies Theorem \ref{thm:zp}, with almost the same bound.
\begin{corollary}
\label{cor:cor1}
Let $k \geq 1$ be an integer and $p$ be a prime. Then for any $A \subseteq \mathbb{Z}_p$ with $|A| < \log_{2k} p - 1$ there exists $A' \subset \mathbb{Z}$ $F_k$-isomorphic with $A$ via the canonical homomorphism.
\end{corollary}
\begin{proof}
We consider all linear polynomials in $n := |A|$ variables having $\|\cdot\|_1$-norm at most $2k$, and split them into $\mathcal{L}_1$ and $\mathcal{L}_2$ according to the result of evaluation with elements from $A$. This includes all polynomials used in the definition of the usual Freiman isomorphism. Applying Lemma \ref{lem:additive}, we get a subset $A' \subset \mathbb{Z}_{(p)}$, which by definition must be $F_k$-isomorphic with $A$ via the canonical homomorphism. Multiplying all values of $A'$ by a large enough integer, which is $1$ modulo $p$ and cleares all denominators, will ensure that $A'$ lies in $\mathbb{Z}$, while still being $F_k$-isomorphic with $A$ via the canonical homomorphism.
\end{proof}
\begin{proof}[Proof of Lemma \ref{lem:additive}]

We can express $\mathcal{L}_1$ as the system $M\mathbf{x} = \mathbf{b}$, for some $m \times n$ matrix $M$ and vector $\mathbf{b}$. We then form the augmented matrix $M' = (M | \mathbf{b})$. By assumption, the $\|\cdot\|_1$-norm of any row of $M'$ is at most $k$.

The system $\mathcal{L}_1$ is solvable in a field $\mathbb{F}$ if and only if $\rank_{\mathbb{F}} M = \rank_{\mathbb{F}} M'$. We will show that this is the case in $\mathbb{Q}$.

As the rank of $M'$ is the maximum size of one of its square submatrices with non-zero determinant, we see that $\rank_{\mathbb{Q}} M' \geq \rank_{\mathbb{F}_p} M'$. On the other hand, let $M_1'$ be any square submatrix of $M'$ of full rank in $\mathbb{Q}$. By Lemma \ref{lem:ineq}, $|\det(M_1')| \leq k^{n+1} < p$. Hence $\det(M_1')$ is also non-zero in $\mathbb{F}_p$, and consequently $\rank_{\mathbb{Q}} M' \leq \rank_{\mathbb{F}_p} M'$. But then $M'$ has the same rank $t$ in $\mathbb{Q}$ and in $\mathbb{F}_p$. Similarly we obtain that $M$ has the same rank in both $\mathbb{Q}$ and $\mathbb{F}_p$. However, the system $\mathcal{L}_1$ is solvable in $\mathbb{F}_p$, and so we must have $t = \rank M \leq n$. Consequently $\mathcal{L}_1$ is solvable in $\mathbb{Q}$. This is nevertheless not enough for our purposes; we must further show that a solution $A'$ with the desired properties exists.

We may assume w.l.o.g. that
\begin{equation*}
M = \left( 
\begin{array}{cc}
M_1 & M_2 \\
M_3 & M_4 
\end{array} \right), \quad \mathbf{b} = \left( \begin{array}{c} \mathbf{b}_1 \\ \mathbf{b}_2 \end{array} \right) 
\end{equation*}
where $M_1$ is a square matrix of full rank $t = \rank M$ in both $\mathbb{Q}$ and $\mathbb{F}_p$, and $\mathbf{b}$ is partitioned accordingly. Let $M_1^{*}$ be the adjoint of $M_1$.

We get
\begin{equation*}
\left(\begin{array}{cc}
M_1^{*} & 0 \\
0 & I 
\end{array}\right)
\left( 
\begin{array}{cc}
M_1 & M_2 \\
M_3 & M_4 
\end{array} \right) x
= \left(\begin{array}{cc}
\det(M_1)I & M_1^{*}M_2 \\
M_3 & M_4 
\end{array}\right) x = \left( \begin{array}{c} M_1^{*}b_1 \\ b_2 \end{array} \right).
\end{equation*}
By Lemma \ref{lem:ineq}, $|\det(M_1)| \leq k^n$.

Consequently we can express the first $t$ variables in terms of the last $n-t$ variables, involving fractions with denominator bounded by $k^n < p$. By letting $b_i := a_i$ and replacing $x_i$ with $b_i$ in these equations for $t+1 \leq i \leq n$, we obtain values $b_1, \ldots, b_t$ in $\mathbb{Z}_{(p)}$ for $x_1, \ldots, x_t$ such that $b_i$ is mapped to $a_i$ by the canonical homomorphism, for any $1 \leq i \leq n$. Furthermore, as $\rank_{\mathbb{Q}} M' = t$, by replacing $x_i$ with $b_i$ in the last $m-t$ equations we obtain the identity $0=0$ in $\mathbb{Q}$ everywhere. 

We conclude that $A' := \{b_1, \ldots, b_n\}$ is a solution for $\mathcal{L}_1$ in $\mathbb{Z}_{(p)}$. Furthermore, no polynomial $f \in \mathcal{L}_2$ can be zero when evaluated at $A'$, for otherwise it would also be zero modulo $p$, hence $0$ when evaluated at $A$, a contradiction. Then we are done.
\end{proof}

\section{Resultants, subresultants and the gcd}
\label{sec:result}

As in the case of linear polynomials, we must bound the complexity of solving a system of multivariate polynomials. We gather in this section all the tools required for the proof.

In what follows we shall introduce and make substantial use of subresultants, an alternative to Euclid's algorithm for computing the greatest common divisor of two polynomials.
This approach will be essential in obtaining any reasonable quantitative bound in Theorem \ref{thm:main}, as Euclid's algorithm leads to an explosive growth of the coefficients involved in the polynomial division.

Suppose $A$ is an integral domain. If $A \subseteq B$, $B$ is a commutative ring and $b \in B$, we shall denote by $\ev_b$ the evaluation homomorphism $\ev_b : A[x] \rightarrow B$ mapping $f(x)$ to $f(b)$. If $0 \neq a \in A$, we shall denote by $A[\frac{1}{a}]$ the ring of polynomials $A[x]$ evaluated at $\frac{1}{a}$. This is the same as the ring of fractions of $A$ with respect to $\{a^n : n \geq 0\}$, and is sometimes denoted by $A_a$. If $B$ is another integral domain and $\phi : A \rightarrow B$ is a homomorphism, $\phi$ extends to a homomorphism from $A[x]$ to $B[x]$, which we shall also denote by $\phi$.

Let $f, g \in A[x]$. We say $g | f$ if there exists $h \in A[x]$ such that $f = hg$. Hence $h | 0$ for any $h \in A[x]$, but $0$ divides only $0$. Moreover if $A$ is a unique factorization domain (UFD), then $\gcd_A(f, g)$ is well-defined. Here we use the conventions $\gcd_A(h, 0) = \gcd_A(0, h) = h$, for any polynomial $h$. Note that $\gcd_A(f, g)$ is unique only up to a unit of $A$. If no confusion may occur, we shall drop the subscript $A$. Furthermore if $f_1, \ldots, f_m \in A[x]$ we let $\gcd(f_1, \ldots, f_m)$ denote their greatest common divisor, where for $m=1$ this is by convention $f_1$.

We also make the convention $\deg(0) = -\infty$.

We shall need the following easy fact.

\begin{lemma}
\label{lem:ringdiv}
Suppose $A \subseteq B$ are integral domains, $f, g \in A[x]$ non-zero and $g| f$ in $B[x]$. Then $g | f$ in $A[\frac{1}{\gamma}]$, where $\gamma$ is the leading coefficient of $g$.
\end{lemma}
\begin{proof}
By replacing $A$ with $A[\frac{1}{\gamma}]$ and $B$ with $B[\frac{1}{\gamma}]$, we may suppose $\frac{1}{\gamma} \in A$.

Assume $p := \deg(f), q := \deg(g)$ and $a \neq 0$ is the leading coefficient of $f$. By assumption, $f = hg$, for some $h \in B[x]$.

We prove by induction on $\deg(h) \geq 0$ that $h \in A[x]$.

Let $c \neq 0$ be the leading coefficient of $h$. Note that $\deg(h) = p - q$. Then $c\gamma = a$, and so $c = \frac{a}{\gamma} \in A$. If $\deg(h) = 0$, we are done, otherwise 
$f - cx^{p-q}g = (h-cx^{p-q})g$, and so by induction $h - cx^{p-q} \in A[x]$. Thus the claim is
proved.
\end{proof}

Now let $A$ be an integral domain, $f, g \in A[x]$ be non-zero polynomials and suppose $f = a_px^p + \ldots + a_0, g = b_qx^q + \ldots + b_0$ with $a_p, b_q \neq 0$. The \textit{Sylvester matrix} of $f$ and $g$ is the $(p+q) \times (p+q)$ matrix 
\begin{equation*}
S_{f, g} := \left(\begin{array}{ccccc}
a_p & \ldots & a_0 & & \\
& \ddots & & \ddots & \\
& & a_p & \ldots & a_0 \\
b_q & \ldots & b_0 & & \\
& \ddots & & \ddots & \\
& & b_q & \ldots & b_0 
\end{array}\right),
\end{equation*}
where the first $q$ lines are formed by shifting the first row to the right, and the last $p$ lines are formed by shifting the $(q+1)$th row to the right. If $p=q=0$, we define $S_{f,g} = (1)$. The \textit{resultant} of $f$ and $g$, denoted by $\res(f, g)$, is the determinant of $S_{f, g}$. We also define $\res(0, h) = \res(h, 0) = 0$, for any polynomial $h$, so that the resultant is now properly defined for any two polynomials in $A[x]$. 

The main application of resultants is to determine when two polynomials have a common root.
\begin{theorem}[Proposition $4.16$, \cite{BasuPollack}]
\label{thm:root}
Suppose $A$ is a UFD and $f, g \in A[x]$ are non-zero. Then $\gcd(f, g)$ is non-constant if and only if $\res(f, g) = 0$.
\end{theorem}

Unfortunately we will have to deal with more than two polynomials and more than one variable. We therefore make the following definition, following \cite{Hermann1974}.

Let $f_1, \ldots, f_m \in A[x_1, \ldots, x_n], m \geq 1$. Let $y_3, \ldots, y_m$ be new indeterminates and define $A' := A[x_2, \ldots, x_n], A'' := A'[y_3, \ldots, y_m]$. Let $F_1, F_2$ be polynomials in $A''[x_1]$ defined as follows:
\begin{align}
F_1 &:= f_1 \label{eq:F1F2} \\
F_2 &:= f_2 + y_3f_3 + \ldots + y_mf_m.\nonumber
\end{align}
If $m = 1$, we take $F_2 : = 0$. We define the resultant of the polynomials $f_1, \ldots, f_m$ in terms of $x_1$, denoted by $\res_{x_1}(f_1, \ldots, f_m)$, as the resultant of $F_1$ and $F_2$. Note that this is a polynomial in $x_2, \ldots, x_n$ and $y_3, \ldots, y_m$.

We first have a lemma.
\begin{lemma}
\label{lem:gcd}
Suppose $A$ is a UFD. Then $\gcd_{A'}(f_1, \ldots, f_m) = \gcd_{A''}(F_1, F_2)$.
\end{lemma}
\begin{proof}
If $m = 1$, this is true by definition. So assume $m \geq 2$.

By hypothesis $A'$ and $A''$ are both UFD. Now if $g := \gcd_{A'}(f_1, \ldots, f_m)$ and $g' := 
\gcd_{A''}(F_1, F_2)$ then $g | g'$, as $g | F_1$ and $g | F_2$. Furthermore $g' \in A'$, because $g' | f_1$. Giving values $y_i = 0$ we see that $g' | f_2$. Also if we let $y_j = 1$ and $y_i = 0, i\neq j$, we see that $g' | f_2 + f_j$. Hence $g' | f_j, 2 \leq j \leq m$. Then $g' | g$ and the claim follows.
\end{proof}
\begin{theorem}
\label{thm:resultant}
Assume $A$ is a field and let $K$ be its algebraic closure. Let $(a_2, \ldots, a_n) \in K^{n-1}$ and suppose that the leading coefficient of $x_1$ in $f_1 \in A[x_1, \ldots, x_n]$, a polynomial in $x_2, x_3, \ldots, x_n$, does not vanish when replacing $x_2$ with $a_2$, $x_3$ with $a_3, \ldots, x_n$ with $a_n$. Then there exists an $a_1 \in K$ such that $(a_1, \ldots, a_n)$ is a common zero for $f_1, \ldots, f_m$ if and only if $\res_{x_1}(f_1, \ldots, f_m)(a_2, \ldots, a_n) = 0$.
\end{theorem}
\begin{proof}
We replace $x_i$ by $a_i, 2 \leq i \leq n,$ in $F_1$ and $F_2$. Then the degree of $F_1$ stays the same, but the degree of $F_2$ may decrease with some amount $r \geq 0$. 

If $F_2 = 0$ then by definition $\res_{x_1}(f_1, \ldots, f_m)(a_2, \ldots, a_n) = 0$. As $\deg_{x_1}(F_1) \geq 1$ and $a_1$ can be taken to be any root of $F_1$, the claim trivially holds.

So assume $F_2 \neq 0$. By definition of the Sylvester matrix we know that
\begin{equation*}
\res_{x_1}(f_1, \ldots, f_m)(a_2, \ldots, a_n) = c^r\,\res_{x_1}(f_1(a_2, \ldots, a_n), \ldots, f_m(a_2, \ldots, a_n)).
\end{equation*}
where $0 \neq c \in K$ is the leading  coefficient of $x_1$ in $f_1(a_2, \ldots, a_n)$. Thus by replacing $f_i$ with $f_i(a_2, \ldots, a_n), 1 \leq i \leq m$, and $A$ with $K$, we may suppose w.l.o.g. that $n=1$.

By Lemma \ref{lem:gcd}, $\gcd(f_1, \ldots, f_m) = \gcd(F_1, F_2)$, and hence $a_1$ exists iff $\gcd(F_1, F_2)$ is non-constant. But by Theorem \ref{thm:root} this happens iff $\res(F_1, F_2) = \res_{x_1}(f_1, \ldots, f_m)$ is zero, hence the claim holds.
\end{proof}
For a different proof of Theorem \ref{thm:resultant} see Theorem $6.1$, \cite{Hermann1974}.

We now turn to subresultants.

Let $A$ be an integral domain, $f, g \in A[x]$ non-zero as before and again suppose $f = a_px^p + \ldots + a_0, g = b_qx^q + \ldots + b_0$ with $a_p, b_q \neq 0$. The \textit{subresultant sequence} for $f$ and $g$ is a list of polynomials $S_i(f, g) := \sum_{j=0}^{i}s_{ij}(f, g)x^j, 0 \leq i \leq \min\{p, q\}$, where $s_{ij}(f, g)$ is the determinant of the matrix built with rows $1, \ldots, q- i$ and $q+1, \ldots, q+p-i$ of $S_{f, g}$, and columns $1, 2, \ldots, p+q-2i-1, p+q-i-j$ of $S_{f, g}$. This is well-defined except when $i = p =q$. Thus when $p=q \neq 0$ we set $S_q(f, g) = g$ and define $s_{qj}$ in the obvious way. For $p = q = 0$ we set $S_0(f, g) = 1$.

Due to technical reasons we define the subresultant sequence also for the case when one of $f$ or $g$ (but not both) is $0$. If $g = 0$, we let $S_i(f, g) := S_i(f, f), 0 \leq i \leq \deg(f)$, and we proceed similarly if $f = 0$.

We now have the following result.
\begin{theorem}
\label{thm:sub}
Suppose $A$ is a UFD and $f, g \in A[x]$ are not both zero. If $k \geq 0$ is minimal such that $s_{kk}(f, g) \neq 0$ then there exists non-zero $u, v \in A$ such that $u \gcd(f, g) = v S_k(f, g)$.
\end{theorem}
In a similar form, Theorem \ref{thm:sub} was already known in the 19th Century. Collins \cite{collins67} introduced the terminology of subresultants, leading to the modern formulation of Theorem \ref{thm:sub}, in conjuction with the problem of efficiently computing the gcd of two polynomials. The theory was subsequently refined and simplified by Brown and Traub \cite{Brown71}. A good exposition of the theory of subresultants and a proof of Theorem \ref{thm:sub} can be found in \cite{BasuPollack} (see also \cite{collins66} and \cite{Brown71}). 

In the proof of the main result we will encounter rings which are not UFD, and so we will not be able to apply Theorem \ref{thm:sub} directly. We deal with this situation below.

Let $A$ be an integral domain and $f_1, \ldots, f_m \in A[x], m \geq 1$. We define $F_1$ and $F_2$ as in \eqref{eq:F1F2}. We first make a simple observation.
\begin{lemma}
\label{lem:roottr}
Assume $A \subseteq K \subseteq \overline{K}$, where $K, \overline{K}$ are fields, and $\overline{K}$ is algebraically closed. Suppose $G := \gcd_K(f_1, \ldots, f_m)$ has degree $\delta \geq 1$, and let $b_1, \ldots, b_d$ be the distinct roots of $G$ in $\overline{K}$, each appearing with multiplicity $\mu_i, 1 \leq i \leq d$. Then
\begin{equation}
\label{eq:poldecomp}
S_\delta(F_1, F_2) = \ell \prod_{i=1}^d (x-b_i)^{\mu_i},
\end{equation}
where $\ell$ is the leading coefficient of $S_\delta(F_1, F_2)$ as a polynomial in $x$.
\end{lemma}
As $\delta \geq 1$ we have $\deg(F_1) = \deg(f_1) \geq 1$ and so $S_i(F_1, F_2)$ is well-defined (nevertheless it may happen that $F_2$ is $0$ if $m = 1$). Further recall that $S_{i}(F_1, F_2)$ is a polynomial in $y_3,\ldots,y_m$ and $x$.
\begin{proof}[Proof of Lemma \ref{lem:roottr}]
By Lemma \ref{lem:gcd}, $G = \gcd_{K[y_3, \ldots, y_m]}(F_1, F_2)$. Hence by Theorem \ref{thm:sub}, there are non-zero $u, v \in K[y_3, \ldots, y_m]$ such that $uG = vS_\delta(F_1, F_2)$. But for any $1 \leq i \leq d$, $(x-b_i)^{\mu_i} | uG$ in $\overline{K}[y_3, \ldots, y_m, x]$. Hence $(x-b_i)^{\mu_i} | S_\delta(F_1, F_2), 1 \leq i \leq d$. As $S_\delta(F_1, F_2)$ has degree exactly $\delta$ as a polynomial in $x$, \eqref{eq:poldecomp} must hold, thus proving the lemma.
\end{proof}

The main consequence of Theorem \ref{thm:sub} is the following.

\begin{lemma}
\label{lem:diagram}
Suppose $A \subseteq \mathbb{C}, G := \gcd_{\mathbb{C}}(f_1, \ldots, f_m)$ has degree $\delta \geq 1$, $\ell := s_{\delta\delta}(F_1, F_2)$ and  $\phi:A \rightarrow \mathbb{F}_p$ is a homomorphism such that
\begin{equation}
\label{eq:condphi}
\deg_x(\phi(F_1)) = \deg_x(F_1), \quad \deg_x(\phi(F_2)) = \deg_x(F_2)
\quad {\textrm{and}} \quad \phi(\ell) \neq 0.
\end{equation}
Then for any root $b' \in \mathbb{F}_p$ of $\gcd_{\mathbb{F}_p}(\phi(f_1), \ldots, \phi(f_m))$ there exists a root $b$ of $G$ and a homomorphism $\Phi : A[b] \rightarrow \mathbb{F}_p$ such that the following diagram commutes
\begin{equation}
\label{eq:diagram}
\begin{CD}
A[x] @>{\ev_{b}}>> A[b]\\
@V{\phi}VV					@VV{\Phi}V\\
\mathbb{F}_p[x] @>{\ev_{b'}}>> \mathbb{F}_p
\end{CD}
\end{equation}
\end{lemma}
\begin{proof}
By definition the case $m=1$ is equivalent to the case $m = 2$ where $f_2 = f_1$, and so we will assume w.l.o.g. that $m \geq 2$ and $F_2 \neq 0$. Let $G' := \gcd_{\mathbb{F}_p}(\phi(f_1), \ldots, \phi(f_m))$.

As $\deg_x(\phi(F_1)) = \deg_x(F_1)$ and $\deg_x(\phi(F_2)) = \deg_x(F_2)$, we have $\phi(S_i(F_1, F_2)) = S_i(\phi(F_1), \phi(F_2))$. Hence by Theorem \ref{thm:sub} and the fact that $\phi(\ell) \neq 0$, we have $\deg(G) = \deg(G') = \delta \geq 1$. 

Let $b'$ be any root of $G'$ in $\mathbb{F}_p$. By Lemma \ref{lem:roottr} we have $\phi(S_{\delta}(F_1, F_2))(b') = 0$. Define $\psi := ev_{b'} \circ \phi : A[x] \rightarrow \mathbb{F}_p$.

Let $b_1, \ldots, b_d$ be the distinct roots of $G$ in $\mathbb{C}$, each appearing with multiplicity $\mu_i, 1\leq i \leq d$. Assume for a contradiction that for any root $b_i$ of $G$ there is no homomorphism $\Phi$ making the diagram \eqref{eq:diagram} commutative. This means $\ker \ev_{b_i} \not\subseteq \ker \psi$, so there exists a polynomial $g_i \in A[x]$ such that $g_i(b_i) = 0$, but $(\phi \circ g_i)(b') \neq 0$.

Define 
\begin{equation*}
H := \ell \prod_{i=1}^d g_i^{\mu_i}.
\end{equation*}
Then $H \in A[x, y_3, \ldots, y_m]$. As $\phi(\ell) \neq 0$, we have $\phi(H)(b') \neq 0$ in $\mathbb{F}_p[y_3, \ldots, y_m]$. But by Lemma \ref{lem:roottr}, 
\begin{equation*}
S_\delta(F_1, F_2) = \ell\prod_{i=1}^d (x - b_i)^{\mu_i}
\end{equation*}
in $\mathbb{C}[x, y_3, \ldots, y_m]$. Then $S_\delta(F_1, F_2) | H$ in $\mathbb{C}[x, y_3, \ldots, y_m]$. Hence by Lemma \ref{lem:ringdiv}, $S_\delta(F_1, F_2) | H$ in $A[x, y_3, \ldots, y_m, \frac{1}{\ell}]$. 
But $\phi(\ell) \neq 0$, so $\phi$ extends to a homomorphism 
\begin{equation*}
\phi: A[x, y_3, \ldots, y_m, \frac{1}{\ell}] \rightarrow \mathbb{F}_p[x, y_3, \ldots, y_m].
\end{equation*}
This implies $\phi(S_\delta(F_1, F_2)) | \phi(H)$. As $\phi(S_\delta(F_1, F_2))(b') = 0$, we obtain $\phi(H)(b') = 0$, a contradiction. This finishes the proof of the lemma.
\end{proof}

\section{Preserving both the additive and multiplicative structure}
\label{sec:main}

We have the following technical result.
\begin{lemma}
\label{lem:multipl}
Let $k, t \geq 2$ be integers and $p$ be a prime. Suppose $A = \{a_1, \ldots, a_n\} \subseteq \mathbb{F}_p$ and let $\mathcal{L}_1, \mathcal{L}_2 \subset \mathbb{Z}[x_1, \ldots, x_n]$ be collections of $(k, t)$-bounded polynomials, such that any $f \in \mathcal{L}_1$ is zero when evaluated at $(a_1, \ldots, a_n)$, and any $f \in \mathcal{L}_2$ is non-zero when evaluated at $(a_1, \ldots, a_n)$.  If
\begin{equation}
|A| < \log_2\log_{2t}\log_{2kt}p -1
\end{equation}
then there exists a finite algebraic extension $K$ of $\mathbb{Q}$ of degree at most $(2t)^{2^n}$ and a subset $A' = \{b_1, \ldots, b_n\} \subset K$ such that $f(b_1, \ldots, b_n) = 0$ for $f \in \mathcal{L}_1$, and $f(b_1, \ldots, b_n) \neq 0$ for $f \in \mathcal{L}_2$. Furthermore, the map $\phi_p: \mathbb{Z}[A'] \rightarrow \mathbb{F}_p$ sending $b_i$ to $a_i$ is a ring homomorphism.
\end{lemma}
\begin{proof}
We first give a rough overview of the proof.

The proof has three steps. 

In the first step we eliminate the variables one by one. We start with the collection of polynomials $\mathcal{L}^0 := \mathcal{L}_1$ and we compute the resultant $R_1$ in terms of $x_1$. We then form a new collection of polynomials $\mathcal{L}^1$ in $x_2, \ldots, x_n$ by taking the coefficients of the $y$-monomials in $R_1$. By Theorem \ref{thm:resultant}, there is at least one choice for $x_1$ iff there exists a common solution to the polynomials in $\mathcal{L}^1$. We then eliminate $x_2$ and proceed further in the same manner to construct collections $\mathcal{L}^i$. After at most $n$ steps we have eliminated all variables, and only constant polynomials remain. However, the same procedure could have been carried over in $\mathbb{F}_p$, with the same starting collection of polynomials, and there it is guaranteed that a solution exists. Hence if the final constants are less than $p$, they must in fact be $0$, and so a solution exists in $\mathbb{C}$ as well.

In the second step we go back, trying to determine the $b_i$'s. Suppose for example that we have only polynomials in one variable, say $x_n$, and we know that a common root exists. Then their gcd is non-constant, and we can use Lemma \ref{lem:diagram} to pick one of the roots of the gcd as $b_n$. The hypothesis of Lemma \ref{lem:diagram} will be satisfied by adding some more polynomials to $\mathcal{L}^i$ in the first step. We then adjoin $b_n$ to $\mathbb{Q}$, replace $x_n$ by $b_n$, and proceed similarly to determine $b_{n-1}$. Theorem \ref{thm:resultant} will ensure that once $b_{i+1}, \ldots, b_n$ are picked, there is still a choice for $b_i$.

Note that once the homomorphism $\phi_p$ is constructed, the conditions imposed by $\mathcal{L}_2$ are automatically satisfied. For if $f \in \mathcal{L}_2$ then $\phi_p(f(b_1, \ldots, b_n)) = f(a_1, \ldots, a_n) \not\equiv 0\,(\textrm{mod } p)$, hence $f(b_1, \ldots, b_n) \neq 0$ as well.

In the last step we will estimate the degree of the extension.

We now present the proof in detail.

\medskip
\noindent
\textbf{\textit{Step 1}}. We let $u_0 := k, v_0 := t$ and for any $1 \leq i \leq n$ we define $u_i$ and $v_i$ inductively by 
\begin{align*}
u_i &:= u_{i-1}^{2v_{i-1}}v_{i-1}^{v_{i-1}},\\
v_i &:= 2v_{i-1}^2.
\end{align*}
We shall prove in Step 3 that for $0 \leq i \leq n$ we have
\begin{equation}
\label{eq:mainCond}
u_i < p.
\end{equation}
Assume for the moment that this is indeed the case. For $0 \leq i \leq n$ let $\sigma_i : \mathbb{Z}[x_{i+1}, \ldots, x_n] \rightarrow \mathbb{F}_p[x_{i+1}]$ be the homomorphism mapping $x_j$ to $a_j, i+1 < j \leq n$. We similarly define $\sigma : \mathbb{Z}[x_1, \ldots, x_n] \rightarrow \mathbb{F}_p$ as the homomorphism mapping $x_j$ to $a_j$ for all $1 \leq j \leq n$.

We will construct by induction on $i \geq 0$ sets $\mathcal{L}_1 = \mathcal{L}^0, \mathcal{L}^1, \ldots, \mathcal{L}^r, r \leq n,$ such that $\mathcal{L}^i \subset \mathbb{Z}[x_{i+1}, \ldots, x_n]$ is a collection of $(u_i, v_i)$-bounded polynomials satisfying $\sigma(f) = 0$ for any $f \in \mathcal{L}^i, 0 \leq i \leq r$. Furthermore, it will be necessary at every step $i < r$ to slightly modify the set $\mathcal{L}^i$ into another one $\mathcal{A}_i$ by altering some of the polynomials. $\mathcal{A}_i$ will still contain only $(u_i, v_i)$-bounded polynomials $f$ verifying $\sigma(f) = 0$.

The construction of the sets $\mathcal{L}^i$ will be done in three stages, indicated by the bold letters \textbf{(A), (B)} and \textbf{(C)}.

For $i = 0$, by assumption $\mathcal{L}^0$ is a collection of $(u_0, v_0)$-bounded polynomials mapped to $0$ by $\sigma$.

Now suppose $n \geq i \geq 0$ and we have constructed $\mathcal{L}^i$. If $i = n$ or $\mathcal{L}^i$ is empty or $\{0\}$, we set $r = i$ and stop. Otherwise, let $\mathcal{L}^i = \{f_1, \ldots, f_m\}$ and $f_j = \sum_{\ell=0}^{d_j} c_{j\ell}x_{i+1}^\ell$. By assumption we have $i \leq n-1$.

\textbf{(A)} For any $1 \leq j \leq m$ and $\deg_{x_{i+1}}(\sigma_i(f_j)) < \ell \leq \deg_{x_{i+1}}(f_j)$ we put $c_{j\ell}$ into $\mathcal{L}^{i+1}$.

We then set $d_j' = \deg_{x_{i+1}}(\sigma_i(f_j))$ and define 
\begin{equation*}
f_j' := \sum_{\ell=0}^{d_j'} c_{j\ell}x_{i+1}^\ell.
\end{equation*}

Note that $d_j' \neq 0$, otherwise $\sigma(f_j) = \sigma_i(f_j) \neq 0$, a contradiction.

Let $\mathcal{A}_i := \{f_1', \ldots, f_m'\}$. Clearly every polynomial in $\mathcal{A}_i$ is still $(u_i, v_i)$-bounded. Furthermore if $x_{i+1}$ does not appear in any polynomial in $\mathcal{A}_i$, then $\mathcal{A}_i$ contains only $0$ by the above. In this case there is nothing else to be done.

So assume w.l.o.g. that $x_{i+1}$ appears in $f_1'$. Let
\begin{align*}
F_1 &:= f_1' \\
F_2 &:= f_2' + y_3f_3' + \ldots + y_mf_m'
\end{align*}
for unknowns $y_3, \ldots, y_m$, where $F_2 := 0$ if $m=1$. 

\textbf{(B)} We take $\res_{x_{i+1}}(f_1', \ldots, f_m')$ and put into $\mathcal{L}^{i+1}$ the coefficient of every monomial in $y_j$, which must be a polynomial in $x_{i+2}, \ldots, x_n$.

Set $R_1 := \mathbb{Z}[x_{i+2}, \ldots, x_n, y_3, \ldots, y_m]$ and $R_2 := \mathbb{F}_p[y_3, \ldots, y_m]$. Note that $\sigma_i$ induces a homomorphism between $R_1[x_{i+1}]$ and $R_2[x_{i+1}]$. We have $F_1, F_2 \in R_1[x_{i+1}]$ and by (A), $\deg_{x_{i+1}}(\sigma_{i}(F_1)) = \deg_{x_{i+1}}(F_1)$ and $\deg_{x_{i+1}}(\sigma_{i}(F_2)) = \deg_{x_{i+1}}(F_2)$. So let $q_1 := \deg_{x_{i+1}}(F_1)$ and $q_2 :=\deg_{x_{i+1}}(F_2)$. By assumption, $q_1 \geq 1$.

Let $\delta \geq 0$ be minimal such that $\sigma_{i}(s_{\delta\delta}(F_1, F_2)) \neq 0$, where $s_{k\ell}$ are the coefficients of the subresultant sequence. 

\textbf{(C)} We put into $\mathcal{L}^{i+1}$ the coefficients of $s_{jj}(F_1, F_2)$ (polynomials in $x_{i+2}, \ldots, x_n$), for $1 \leq j < \delta$. For $j=0$ this has already been done, as $s_{00}(F_1, F_2) =  \res_{x_{i+1}}(f_1', \ldots, f_m')$ by definition.

The construction of $\mathcal{L}^{i+1}$ is now over. We must show that any polynomial in $\mathcal{L}^{i+1}$ is indeed $(u_{i+1}, v_{i+1})$-bounded.

This is certainly the case for the polynomials added in stage (A). So consider the stage (B) of the construction. 
 
Fix an arbitrary monomial $M$ in $y_3, \ldots, y_m$ of degree at most $q_1$. This has a coefficient $g$ in $\res_{x_{i+1}}(f_1', \ldots, f_m')$ and we must estimate $\|g\|_1$ and $\deg(g)$. Since $q_1, q_2 \leq v_i$, the degree of $g$ is at most $(q_1+q_2)v_i \leq 2v_i^2 = v_{i+1}$, as desired.

Now let $2 \leq j_1, j_2, \ldots, j_{q_1} \leq m$ and define $S_{F_1, F_2}(j_1, \ldots, j_{q_1})$ by writing on line $q_2+k'$ of $S_{F_1, F_2}$, instead of the coeficients of $F_2$, the corresponding coefficients of $f_{j_{k'}}', 1 \leq k' \leq q_1$. Then $g$ is a sum of $\det(S_{F_1, F_2}(j_1, \ldots, j_{q_1}))$, for certain $q_1$-tuples $j_1, j_2, \ldots, j_{q_1}$ depending on $M$. The number of such $q_1$-tuples is 
\begin{equation*}
\binom{q_1}{\deg(M)}\frac{\deg(M)!}{\deg_{y_3}(M)!\ldots\deg_{y_m}(M)!} \leq q_1^{\deg(M)} \leq v_i^{v_i}.
\end{equation*}
Recall that $\|f_j'\|_1 \leq u_i, 1 \leq j \leq m$. So by Lemma \ref{lem:ineq} applied to $S_{F_1, F_2}(j_1, \ldots, j_{q_1})$ (a square matrix of size $q_1+q_2 \leq 2v_i$), we obtain $\|\det(S_{F_1, F_2}(j_1, \ldots, j_{q_1}))\|_1 \leq u_i^{2v_i}$. Hence $\| g \|_1 \leq u_i^{2v_i} v_i^{v_i} = u_{i+1}$, as desired.

Finally, as subresultants are defined using submatrices of $S_{F_1, F_2}$, all the above estimates apply to subresultants as well. Hence any polynomial added to $\mathcal{L}^{i+1}$ in stage (C) is also $(u_{i+1}, v_{i+1})$-bounded. Consequently any polynomial in $\mathcal{L}^{i+1}$ is $(u_{i+1}, v_{i+1})$-bounded, as claimed.

We must further check that $\sigma$ maps all the polynomials in $\mathcal{L}^{i+1}$ to $0$. This is certainly the case with the polynomials added in stages (A) and (C) of the construction. As $\deg_{x_{i+1}}(\sigma_i(f_j')) = \deg_{x_{i+1}}(f_j'), 1 \leq j \leq m$, we have that $\res_{x_{i+1}}(\sigma_i(f_1'), \ldots, \sigma_i(f_m')) = \sigma_i(\res_{x_{i+1}}(f_1', \ldots, f_m'))$. By Theorem \ref{thm:resultant} and the fact that the polynomials $\sigma_i(f_j')$ have the common root $a_{i+1}$, we obtain $\res_{x_{i+1}}(\sigma_i(f_1'), \ldots, \sigma_i(f_m')) = 0$. This shows that all the polynomials added in stage (B) of the construction are indeed mapped to $0$ by $\sigma$. Thus the induction step is verified.

\medskip
\noindent
\textbf{\textit{Step 2}}. If $\mathcal{L}^r$ is empty, all the sets $\mathcal{L}^i$ were empty, in particular $\mathcal{L}^0 = \mathcal{L}_1 = \emptyset$. Then we set $b_i = a_i, 1 \leq i \leq n$, take $\phi_p$ to be the canonical homomorphism, and we are done.

So we may assume that $\mathcal{L}^r$ is non-empty. Let $f \in \mathcal{L}^r$. By construction $f$ is an integer constant at most $u_r$ in absolute value, and $u_r < p$ by \eqref{eq:mainCond}. However, $\sigma(f) = 0$, and as $\sigma$ is a homomorphism, we must have $f = 0$. Hence $\mathcal{L}^r = \{0\}$.

By decreasing induction on $r \geq i \geq 0$ we shall find algebraic numbers $b_{i+1}, \ldots, b_n$ such that for any $f \in \mathcal{L}^i$, $f(b_{i+1}, \ldots, b_n) = 0$, and furthermore the map $\phi_p^i : \mathbb{Z}[b_{i+1}, \ldots, b_n] \rightarrow \mathbb{F}_p$, sending $b_j$ to $a_j, i < j \leq n$, is a well-defined homomorphism.

For any $j > r$, we let $b_j$ be the integer in $\{0, 1, \ldots, p-1\}$ satisfying $b_j \equiv a_j\,(\textrm{mod } p)$. Then $\phi_p^r = \sigma\big|_{\mathbb{Z}}$ is a homomorphism. As $\mathcal{L}^r = \{0\}$, the base case $i = r$ is verified.

Now assume $0 \leq i < r$ and we have found $b_{i+2}, \ldots, b_n$ satisfying the induction hypothesis. 

Suppose $\mathcal{L}^i = \{f_1, \ldots, f_m\}$ and $\mathcal{A}_i = \{f_1', \ldots, f_m'\}$. We replace $x_{i+2}, \ldots, x_n$ with their values $b_j$ in the polynomials $f_1, \ldots, f_m$ and $f_1', \ldots, f_m'$. By (A), $f_j = f_j'$ and furthermore $\deg_{x_{i+1}}(\phi_p^{i+1}(f_j)) = \deg_{x_{i+1}}(f_j), 1 \leq j \leq m$. If $x_{i+1}$ does not appear in any of these polynomials, then all of them are in fact $0$. In this case we let $b_{i+1}$ be the integer in $\{0, 1, \ldots, p-1\}$ satisfying $b_{i+1} \equiv a_{i+1}\,(\textrm{mod } p)$. We have $\phi_p^{i+1}(b_{i+1}) = a_{i+1}$. Thus $\phi_p^i = \phi_p^{i+1}$ is a well-defined homomorphism, and the claim holds.

So assume $x_{i+1}$ appears in $f_1$. Here we use the same indexing scheme as in Step 1; in particular, $f_1$ corresponds to the polynomial $f_1'$ selected in Step 1.

By (B) and Theorem \ref{thm:resultant}, at least one choice $b_{i+1}$ for $x_{i+1}$ exists, such that replacing $x_{i+1}$ with this value vanishes all polynomials in $\mathcal{L}^i$. In other words, $G := \gcd_\mathbb{C}(f_1, \ldots, f_m)$ has degree $\delta \geq 1$.

Now recall our construction of $F_1$ and $F_2$. By (A), $\deg_{x_{i+1}}(\phi_p^{i+1}(F_2)) = \deg_{x_{i+1}}(F_2)$. Let $\ell := s_{\delta \delta}(F_1, F_2)$. By (C), Lemma \ref{lem:gcd} and Theorem \ref{thm:sub} applied to $F_1$ and $F_2$ in $\mathbb{C}[x_{i+1}, y_3, \ldots, y_m]$, we see that $\phi_p^{i+1}(\ell) \neq 0$. 

Hence the hypothesis of Lemma \ref{lem:diagram} is satisfied for the ring $A := \mathbb{Z}[b_{i+2}, \ldots, b_n]$, the polynomials $f_1, \ldots, f_m$ and the homomorphism $\phi := \phi_p^{i+1}$. This implies that for the root $a_{i+1}$ of $\gcd_{\mathbb{F}_p}(\phi_p^{i+1}(f_1), \ldots, \phi_p^{i+1}(f_m))$ there exists a root $b_{i+1}$ of $G$ and a homomorphism $\phi_p^i : \mathbb{Z}[b_{i+1}, \ldots, b_n] \rightarrow~\mathbb{F}_p$ making the diagram \eqref{eq:diagram} commutative.
Then $\phi_p^i$ still maps $b_j$ to $a_j$ for $i+1 < j \leq n$. Furthermore by construction, replacing $x_{i+1}$ with $b_{i+1}$ in the polynomials in $\mathcal{L}^i$ vanishes all of them. This proves the induction step.

Continuing in this way we obtain all algebraic numbers $b_1, \ldots, b_n$ and in the last step $\phi_p := \phi_p^0$ maps $b_j$ to $a_j$ as desired.

\medskip
\noindent
\textbf{\textit{Step 3}}. We now compute the degree of the extension and verify \eqref{eq:mainCond}.

First note that $r \leq n$ and $v_i = 2^{2^i-1} t^{2^i}, 0 \leq i \leq n$. Then the degree of the extension is at most
\begin{equation*}
\prod_{i=0}^{r-1} v_i \leq \prod_{i=0}^{n-1} 2^{2^i-1} t^{2^i} \leq 2^{2^n - (n+1)}t^{2^n} \leq (2t)^{2^n}.
\end{equation*}
Further note that
\begin{equation*}
\prod_{i=0}^{n-1} 2v_i \leq 2^{2^n-1} t^{2^n}.
\end{equation*}
We also have $u_0 = k$ and
\begin{equation}
\label{eq:iter}
u_{i+1} = u_i^{2v_i}v_i^{v_i}, 
\end{equation}
and so by iterating \eqref{eq:iter}, and using the above estimates, we obtain
\begin{align*}
u_n &= u_{n-1}^{2v_{n-1}}v_{n-1}^{v_{n-1}} \\
&= u_{n-2}^{(2v_{n-2})(2v_{n-1})}v_{n-2}^{v_{n-2}(2v_{n-1})}v_{n-1}^{v_{n-1}} \\
&= \ldots \\
&= \exp\left\{\left(\prod_{i=0}^{n-1}2v_i \right)\log k + \sum_{i=0}^{n-1}v_i(2v_{i+1})\ldots(2v_{n-1})\log v_i\right\} \\
&\leq \exp\left\{\left(\prod_{i=0}^{n-1}2v_i\right)(\log k + n \log v_{n-1})\right\} \\
&\leq \exp\left\{2^{2^n - 1}t^{2^n}(\log k + n \log(2t)^{2^{n-1}})\right\} \\
&\leq k^{2^{2^n}t^{2^n}}(2t)^{n2^{n-1}2^{2^n - 1}t^{2^n}} \\
&\leq k^{(2t)^{2^n}}(2t)^{2^{2^n+2n - 2}t^{2^n}} \\
&\leq k^{(2t)^{2^n}}(2t)^{(2t)^{2^{n+1}}} \\
&\leq (2kt)^{(2t)^{2^{n+1}}}.
\end{align*}

Thus the condition $u_n < p$ is satisfied if $n < \log_2\log_{2t}\log_{2kt}p - 1$. This shows that \eqref{eq:mainCond} holds, and hence the proof is finished.
\end{proof}
 
\begin{proof}[Proof of Theorem \ref{thm:main}]
We consider all $k$-bounded polynomials in $n := |A|$ variables, and we split them into $\mathcal{L}_1$ and $\mathcal{L}_2$ according to the result of evaluation with elements from $A$. Applying Lemma \ref{lem:multipl}, we get a finite algebraic extension $K$ of $\mathbb{Q}$ of degree at most $(2k)^{2^n}$, a subset $A' \subset K$ and a homomorphism $\phi_p: \mathbb{Z}[A'] \rightarrow \mathbb{F}_p$ which by definition is an $F_k$-ring-isomorphism between $A'$ and $A$. This proves the theorem.
\end{proof}

\section{Sharpness of the main result}

In this section we prove Theorem \ref{thm:sharp}. For $k \geq 2, t \geq 1$ we say that a positive integer $r$ is \textit{$(k,t)$-constructible in at most $n$ steps} if there exists a sequence of non-negative integers $0 = a_0, a_1, \ldots, a_m = r, m \leq n,$ such that for any $i \geq 1, a_i = f_i(a_0, \ldots, a_{i-1})$, with $f_i \in \mathbb{Z}[x_0, \ldots, x_{i-1}]$ a $(k, t)$-bounded polynomial. 

The main step is to prove the following lemma.
\begin{lemma}
\label{lem:chain}
Let $k \geq 2$. Any $p \geq 2^{32(k\log_2(16k))^2}$ is $(k, k)$-constructible in at most $\frac{10}{k}\frac{\log_2 p}{\log_2\log_2 p}$ steps, and moreover this is sharp up to a constant not depending on $k$.
\end{lemma}
\begin{proof}
Let $p \geq  2^{32(k\log_2(16k))^2}$ arbitrary. We first note the following inequality:

\begin{equation}
\label{eq:pbound}
\log_2 \log_2 p \geq 2\log_2(k\log_2\log_2 p).
\end{equation}
Indeed, this is true if $\log_2 p \geq k^2(\log_2\log_2 p)^2$, which in turn is true if $\log p \geq \frac{2k^2}{\log 2}(\log \log p)^2$. By derivation this holds whenever $p \geq 2^{32(k\log_2(16k))^2} \geq e^{8(k\log_2(16k))^2}$.

Now set 
\begin{equation*}
s := \left \lceil \log_2\left(\frac{\log_2 p}{k\log_2\log_2 p}\right) \right \rceil \quad \textrm{and} \quad N := \left \lfloor \log_2 p \right \rfloor.
\end{equation*}
Note that $s \geq 1$, as $\log_2 p > k\log_2\log_2 p$ by \eqref{eq:pbound}.

Consider the base-$2$ representation $(b_0b_1\ldots b_N)$ of $p$, with $b_0$ being the least significant bit. We break it into $\ell := \left\lceil \frac{N+1}{sk} \right\rceil \geq 1$ contiguous subsequences $(b_0b_1\ldots b_{sk-1}),  \ldots, (b_{(\ell-1)sk}b_{(\ell-1)sk+1}\ldots b_N)$, all of them except possibly the last one of length $sk$, defining in base-$2$ numbers $p_0, p_1, \ldots, p_{\ell-1}$. Note that
\begin{equation*}
p = \sum_{i=0}^{\ell-1}2^{ski} p_i
\end{equation*}
and $p_i < 2^{sk}, 0 \leq i < \ell$. We further write
\begin{equation*}
p_i = \sum_{j=0}^{k-1}2^{sj} p_{ij},
\end{equation*}
with $0 \leq p_{ij} < 2^s$.

We now define the sequence $a_0, \ldots, a_{2^s+\ell+2(\ell-1)}$ as follows.

We start by setting $a_0 := 0$ and $a_i := a_{i-1} + 1, 1 \leq i \leq 2^s$. Note that $a_i = i, 1 \leq i \leq 2^s$. For any $0 \leq i \leq \ell-1$ we define
\begin{equation*}
a_{2^s+1+i} := \sum_{j=0}^{k-1}a_{2^s}^ja_{p_{ij}}.
\end{equation*}
Hence $a_{2^s+1+i} = p_i$. For any $1 \leq i \leq \ell-1$ we further define $a_{2^s+\ell+ 2(i-1)+1}$ and $a_{2^s+\ell + 2(i-1)+2}$ as follows:
\begin{align*}
a_{2^s+\ell+ 2(i-1)+1} &:= \left\{
\begin{array}{cc}
a_{2^s}^k, & \textrm{if $i = 1$,}\\
a_{2^s+\ell+2(i-2)+1}a_{2^s+\ell+1}, & \textrm{otherwise.}
\end{array}\right.\\
a_{2^s+\ell+ 2(i-1)+2} &:= \left\{
\begin{array}{cc}
a_{2^s+\ell+1}a_{2^s+2} + a_{2^s+1}, & \textrm{if $i = 1$,}\\
a_{2^s+\ell+2(i-1)+1}a_{2^s+i+1} + a_{2^s+\ell+2(i-2)+2}, & \textrm{otherwise.}
\end{array}\right.
\end{align*}
Hence
\begin{align*}
a_{2^s+\ell+ 2(i-1)+1} &= 2^{ski},\\
a_{2^s+\ell+ 2(i-1)+2} &= \sum_{j=0}^{i}2^{skj} p_j.
\end{align*}
In particular, $a_{2^s+\ell+2(\ell-1)} = p$. Hence $p$ is $(k, k)$-constructible in at most $2^s+3\ell - 2$ steps. But
\begin{align*}
2^s + 3\ell - 2 &\leq 2^s+1+3\frac{N+1}{sk} \\
&\leq 2^s+4\frac{N}{sk}, \quad \textrm{as $sk + 3N + 3 \leq 4N$,}\\
&\leq \frac{2}{k}\frac{\log_2 p}{\log_2 \log_2 p} + \frac{4}{k}\frac{\log_2 p}{\log_2\log_2 p - \log_2(k\log_2\log_2 p)}\\
&\leq \frac{10}{k}\frac{\log_2 p}{\log_2 \log_2 p}, \quad \textrm{by \eqref{eq:pbound}.}
\end{align*}
This proves the first part of the lemma. To show that this bound is essentially best possible, we fix $n$ and count the number of positive integers $(k, k)$-constructible in at most $n$ steps. 

First note that for given $\ell \geq 1$, the number of monomials in $\ell$ variables $x_1, \ldots, x_l$ of degree at most $k$ is $\binom{\ell+k}{k} \leq (kl)^k$. Hence the number of $(k,k)$-bounded polynomials in $\ell$ variables is at most $3^k\binom{\ell+k}{k} \leq (3k\ell)^k$, as any such polynomial is a sum of $k$ monomials in $\ell$ variables of degree at most $k$, with coefficients $1,-1$ or $0$. 

Now to any number which is $(k,k)$-constructible in at most $n$ steps corresponds a sequence of $(k,k)$-bounded polynomials $f_1, \ldots, f_m, m \leq n,$ such that $f_i$ is a polynomial in $i$ variables. Thus the number of integers $(k, k)$-constructible in at most $n$ steps is upper bounded by the number of such sequences, which for $n \geq 3k$ is at most
\begin{equation*}
\prod_{i=1}^n(3ki)^k \leq (3k)^{kn} n^{kn} \leq n^{2kn}.
\end{equation*}
However if $p$ is given, then for $n \leq \frac{\log p}{2k\log \log p}$ we have
\begin{equation*}
n^{2kn} \leq \left(\frac{\log p}{2k\log \log p}\right)^{\frac{\log p}{\log \log p}} < p.
\end{equation*}  
Hence not all numbers between $1$ and $p$ are $(k, k)$-constructible in at most $\frac{\log p}{2k\log \log p}$ steps. This finishes the proof of the lemma.
\end{proof}

\begin{proof}[Proof of Theorem \ref{thm:sharp}]
Given $p \geq 2^{32(k-1)^2\log_2^2(16(k-1))}$ a prime number, we apply Lemma \ref{lem:chain} to find a sequence of non-negative integers $0 = a_0, \ldots, a_n = p, n \leq \frac{10}{k-1}\frac{\log_2 p}{\log_2\log_2 p},$ which shows that it is $(k-1, k-1)$-constructible. Let $A' := \{a_0, a_1, a_2, \ldots, a_n\}$. Taking the residues modulo $p$ of the numbers in $A'$ we obtain a set $A \subseteq \mathbb{F}_p$ of size at most $n$. 

Now suppose for a contradiction that there exists an $F_k$-ring-isomorphism $\phi$ of $A$ into an integral domain $R$ of characteristic $0$. 

There is a natural embedding of $\mathbb{Z}$ into $R$, and we can identify $\mathbb{Z}$ with the image of this embedding. Let $x_i \in A$ be the image of $a_i$ in $\mathbb{F}_p, 0 \leq i \leq n$. By induction on $i \geq 0$ we see that $\phi(x_i)$ must equal $a_i$. 

This is certainly the case for $x_0 = 0$. For $i \geq 1$ there exists a $(k-1)$-bounded polynomial $f_i$ such that $a_i = f_i(a_0, \ldots, a_{i-1})$. Hence $f_i(x_0, \ldots, x_{i-1}) - x_i = 0$ in $\mathbb{F}_p$. As this is a $k$-bounded polynomial, it must be preserved by $\phi$. Therefore the induction hypothesis implies $\phi(x_i) = a_i$, as claimed.

However, $A$ has size at most $n$, while $A'$ has size $n+1$. Therefore the image of $\phi$ can not contain the whole of $A'$, a contradiction. This proves the theorem.
\end{proof}

The proof of Lemma \ref{lem:chain} tells us that for given $M \geq 1$ there are only $(\log M)^{O(\log \log \log M)}$ positive integers less than $M$ which are $(2,2)$-constructible in $O(\log \log M)$ steps. Nevertheless any Mersenne prime (a prime $p$ of the form $2^n-1$) is $(2, 2)$-constructible in $O(\log n) = O(\log \log p)$ steps, by using the base-$2$ representation of $n$ and an approach similar to that of Lemma \ref{lem:chain}. Furthermore any Fermat prime (a prime $p$ of the form $2^{2^n}+1$) is $(2, 2)$-constructible in $O(n) = O(\log \log p)$ steps. Thus the existence of infinitely many such primes would imply Conjecture \ref{conj:sharp}. Unfortunately proving or disproving such a statement seems at present to be an unreachable goal.

\section{Concluding remarks}

\noindent
\textbf{Remark 1.} Theorem \ref{thm:sharp} does not cover the case $k=2$, and in fact here I believe, but can not prove, that the correct bound is $\Theta(\log p)$; that is, any subset $A \subseteq \mathbb{F}_p$ of size $O(\log p)$ is $F_2$-ring-isomorphic to a subset of $\mathbb{C}$. Neither the proof of Theorem \ref{thm:zp} nor that of Lemma \ref{lem:additive} properly adapt to this situation, as one would have to work over the multiplicative group $\mathbb{F}_p^*$ of order $p-1$.

\bigskip
\noindent
\textbf{Remark 2.}
Lemma \ref{lem:multipl} implies the following weaker version of Lemma \ref{lem:chang}: under the hypothesis of Lemma \ref{lem:chang}, there exists a solution $(b_1, \ldots, b_n) \in K^n$ to the polynomials $f_1, \ldots, f_s$, where $K$ is a finite algebraic extension of $\mathbb{Q}$ of degree at most $(2t)^{2^n}$. Indeed, suppose each $f_i$ has degree at most $t$ and $\|\cdot\|_\infty$-norm at most $k$. Then each $f_i$ is $(k(nt)^t,t)$-bounded. Fix $A := \{a_1, \ldots, a_n\}$, the coordinates of a complex solution of the system of polynomials $\{f_i : 1 \leq i \leq s\}$. We first apply Theorem \ref{thm:vuwood} in order to find a sufficiently large prime $p$ (compared to $n, k$ and $t$) and a homomorphism $\phi : \mathbb{Z}[A] \rightarrow \mathbb{F}_p$. We then apply Lemma \ref{lem:multipl} to the collections $\mathcal{L}_1 := \{f_1, \ldots, f_s\}$ and $\mathcal{L}_2 := \emptyset$, in order to find a finite algebraic extension $K$ of degree at most $(2t)^{2^n}$, a subset $A' \subset K$ and a map $\psi$ between $\phi(A)$ and $A'$. Then $((\psi \circ \phi)(a_i))_{i=1}^n$ are the coordinates of a solution $(b_1, \ldots, b_n) \in K^n$ of the system of polynomials $\{f_i : 1 \leq i \leq s\}$.

\bigskip
\noindent
\textbf{Remark 3.} In view of Theorem \ref{thm:mysztr} one may ask what is the largest number $n(p)$ of points and lines in $\mathbb{F}_p^2$ for which the upper bound $cn(p)^{4/3}$ on the number of incidences holds. I have only proved $n(p) = \Omega(\log \log \log p)$, and I am not aware of any non-trivial upper bound for this function. 

\bigskip
\noindent
\textit{Acknowledgements.} The author would like to thank the referee for useful comments that improved the presentation of this paper. The author would also like to thank Pierre Simon for his comments, Tibor Szab\'o for useful comments and pointing out \cite{zahl12} to him, Boris Bukh for bringing \cite{zannier09} to his attention, Karol Cwalina and Yury Person for useful discussions, and the rabbits in Mendelssohn-Bartholdy-Park for providing an inspiring working environment during the summer of 2012, when this research was done.

\renewcommand{\bibname}{}

\end{document}